\documentclass[12pt, letterpaper]{article}

\usepackage[utf8]{inputenc}
\usepackage[pdftex,colorlinks=true,linkcolor=blue,citecolor=blue,urlcolor=red,unicode=true,hyperfootnotes=true,bookmarksnumbered]{hyperref}
\usepackage[margin=1truein]{geometry}

\usepackage{amsfonts}

\usepackage{wrapfig}

\usepackage{caption}
\usepackage{subcaption}

\usepackage{amsfonts, amssymb, amsmath, amsthm}
\usepackage{mathtools}
\usepackage{enumerate}
\usepackage{verbatim}

\usepackage{mathrsfs,amsfonts,mathtools}
\usepackage{amsmath}
\usepackage{amssymb}

\usepackage{amsthm}
\usepackage{thmtools}

\theoremstyle{plain}
\newtheorem{theorem}{Theorem}[section]

\newtheorem{claim}[theorem]{Claim}
\newtheorem{conjecture}[theorem]{Conjecture}

\theoremstyle{definition} 

\renewcommand{\epsilon}{\varepsilon}
\renewcommand{\phi}{\varphi}
\renewcommand{\kappa}{\varkappa}

\newcommand{\lbr}[1]{\left|{#1}\right|}

\newcommand{\wt}[1]{\tilde{#1}}

\newcommand{\mc}[1]{\mathcal{#1}}

\newcommand{\wsat}{\operatorname{\mathrm{wsat}}}
\newcommand{\Wsat}{\mathrm{wSAT}}

\title{A short proof of Tuza's conjecture for weak saturation in hypergraphs}

\author{Nikolai Terekhov\thanks{Department of Discrete Mathematics, Moscow Institute of Physics and Technology, Dolgoprudny, Russia; nikolayterek@gmail.com
}}

\begin{document}

\maketitle

\begin{abstract}
    Given an $r$-uniform hypergraph $H$ and a positive integer $n$, the weak saturation number $\wsat(n,H)$ is the minimum number of edges in an $r$-uniform hypergraph $F$ on $n$ vertices such that the missing edges in $F$ can be added, one at a time, so that each added edge creates a copy of $H$. Shapira and Tyomkyn (Proceedings of the \mbox{American} Mathematical Society, 2023) proved Tuza's conjecture on asymptotic behaviour of $\wsat(n, H)$. In this paper we provide a significantly shorter proof of the conjecture.
\end{abstract}

\section{Introduction}

Cellular automata, introduced by von Neumann~\cite{Neumann} following Ulam’s suggestion~\cite{Ulam}, are used to model processes in physics, biology, chemistry, and cryptography. Bollobás~\cite{Bollobas68} introduced graph bootstrap percolation --- a special case of monotone cellular automata and a substantial generalisation of $r$-neighbourhood bootstrap percolation, which has applications in \mbox{physics ---} see, for example,~\cite{Adler,Fontes,Morris}. Given an \mbox{$r$-uniform} hypergraph $H$, an $H$-bootstrap percolation process is a sequence of hypergraphs $F_0 \subset F_1 \subset \ldots \subset  F_m$ such that, for each $i\ge 1$, $F_i$ is obtained from $F_{i-1}$ by adding an edge that creates a new copy of $H$. An \mbox{$r$-uniform} hypergraph $F$ on $n$ vertices is weakly \mbox{$H$-saturated} if there exists an $H$-bootstrap percolation process $F_0 \subset F_1 \subset \ldots \subset F_m=K^r_n$. In this case, we write $F \in \Wsat(n, H)$. The minimum number of edges in such a hypergraph $F$ is denoted by $\wsat(n, H)$ and is called the weak saturation number.

Weak saturation numbers have been extensively studied. In particular, exact values of weak saturation numbers have been determined for cliques~\cite{Frankl82,Kalai84,Kalai85} and complete bipartite graphs with parts of equal size~\cite{Kalai85,KMM}. Moreover, bounds for general graphs $H$ have also been investigated~\cite{TerekhovZhukovskii,AscoliHe}. In this work, we are particularly interested in asymptotic behaviour of $\wsat(n,H)$. For the case of graphs ($r = 2$), Alon~\cite{Alon85} described the asymptotic behaviour of $\wsat(n,H)$ by proving the existence of the limit $\lim_{n\to\infty} \wsat(n,H)/n$.

Tuza~\cite{Tuza91} proved the following generalisation for hypergraphs. For an $r$-uniform hypergraph $H$, define the \emph{sparseness} $s(H)$ as the size of the smallest subset $S \subseteq V(H)$ such that there exists exactly one edge $U \in E(H)$ with $S \subseteq U$; note that $0\le s(H) \le r$ for every non-empty $r$-uniform hypergraph $H$. Tuza proved that
\begin{equation}\label{eq:theta-wsat}
    \wsat(n, H) = \Theta(n^{s(H)-1}),
\end{equation}
and conjectured that there is, in fact, an exact limit, specifically:
\begin{conjecture}\label{thrm:tuza}
    For any $r$-uniform hypergraph $H$ with at least two edges, there exists a constant $C_H > 0$ such that
    \[\wsat(n,H) = C_H\cdot n^{s(H) - 1} (1 + o(1)).\]
\end{conjecture}

Shapira and Tyomkyn~\cite{Shapira23} proved this conjecture by utilising the following result by Rödl~\cite{Rodl85}.

\begin{theorem}\label{thrm:rodl}
    For any $k \ge t \ge 0$ and $\delta > 0$, there exists $N_0(k, t, \delta) \ge k$ such that for any set $X$ of size $\lbr{X}\ge N_0(k,t,\delta)$, there exists a family $\mc{F}_X \subseteq \binom{X}{k}$ of size $\lbr{\mc{F}_X}\le (1+\delta){\binom{\lbr{X}}{t}}/{\binom{k}{t}}$ such that every $A \in \binom{X}{t}$ is contained in some $W \in \mc{F}_X$.
\end{theorem}

They utilised Theorem \ref{thrm:rodl} to combine $r$-uniform hypergraphs from \mbox{$\Wsat(m, H)$} for small $m$ into an $r$-uniform hypergraph from \mbox{$\Wsat(n, H)$} with the desired estimate on asymptotic. To execute this proof strategy, they introduced a supplementary technical tool, the \mbox{$T_{r,h,s}$-template} saturation process. In this paper we also derive Conjecture \ref{thrm:tuza} from Theorem \ref{thrm:rodl} but without the use of the \mbox{$T_{r,h,s}$-template} saturation process, resulting in a streamlined and significantly shorter proof of \mbox{Conjecture \ref{thrm:tuza}}.

\section{Proof of Conjecture~\ref{thrm:tuza}}

Fix an $r$-uniform hypergraph $H$. The existence of at least two edges ensures that $s(H) \ge 1$. Let $v = \lbr{V(H)}$ and $s = s(H)$. The statement of Conjecture~\ref{thrm:tuza} is equivalent to the existence of the limit
\[\lim_{n\to +\infty} \frac{\wsat(n,H)}{\binom{n - v}{s-1}}.\]

From (\ref{eq:theta-wsat}) it follows that
\[\wt{C}_H = \liminf_{n\rightarrow +\infty} \frac{\wsat(n,H)}{\binom{n - v}{s-1}}\]
exists and $\wt{C}_H > 0$.

Fix $\epsilon > 0$. By the definition of $\wt{C}_H$, there exists $m \geq v + s - 1$ such that
\[\wsat(m,H)\le (\wt{C}_H+\epsilon)\cdot \binom{m-v}{s-1}.\]
Let $F_0$ be a hypergraph from $\Wsat(m, H)$ such that $\lbr{F_0}=\wsat(m,H)$.

By Theorem~\ref{thrm:rodl}, there exists $N_0 \ge m - v$ such that for any set $X$ with $\lbr{X}\ge N_0$, there exists a family $\mc{F}_X \subseteq \binom{X}{m - v}$ such that every subset of $X$ of size $s - 1$ is contained in at least one element of $\mc{F}_X$, where
\[\lbr{\mc{F}_X} \le (1+\epsilon)\cdot\frac{\binom{\lbr{X}}{s-1}}{\binom{m-v}{s-1}}.\]
It also follows that any subset of $X$ of size less than $s - 1$ lies inside some element of $\mc{F}_X$, as such a subset can always be extended to the size $s - 1$.

Fix $n \ge N = N_0 + v$. We construct an $r$-uniform hypergraph $F$ on vertex set $[n]$ in the following way. Let $Z = [v]$ and $X = [n] \setminus Z$. For each element $W \in \mc{F}_X$, add to the hypergraph $F$ a copy of $F_0$ on the vertex set $Z \cup W$. We get that:
\[\lbr{E(F)}\le \lbr{\mc{F}_X}\cdot \lbr{E(F_0)}\le (1+\epsilon)\cdot \frac{\binom{\lbr{X}}{s-1}}{\binom{m-v}{s-1}} \cdot (\wt{C}_H+\epsilon)\cdot \binom{m-v}{s-1}=(1+\epsilon)(\wt{C}_H+\epsilon)\cdot\binom{n-v}{s-1}.\]
Given that $\epsilon>0$ can be chosen arbitrarily small, to prove the theorem, it is sufficient to show that $F$ is weakly $H$-saturated.

First, we complete all copies of $F_0$ in $F$ to cliques. By the definition of $\mc{F}_X$, we obtain a hypergraph $\wt{F}$ such that it contains all edges $e \in \binom{[n]}{r}$ satisfying $\lbr{X\cap e} \le s-1$. This is because any such edge belongs to a clique on $W \cup Z$, where $W \in \mc{F}_X$ and $X \cap e \subseteq W$. Thus, to complete the proof of Conjecture \ref{thrm:tuza}, it remains to apply the following claim, which follows as a special case of a claim proved in~\cite[Theorem 2]{Tuza91}.

\begin{claim}
    Let $H$ be an $r$-uniform hypergraph, and let $F$ be an $r$-uniform hypergraph whose vertices can be partitioned into two sets $Z$ and $X$, where $\lbr{Z}\ge \lbr{V(H)}$ and $F$ contains all edges $e$ such that $\lbr{X\cap e}\le s(H) - 1$. Then $F$ is weakly $H$-saturated.
\end{claim}

\begin{proof}

For completeness, we provide the proof.

We will show that all missing edges can be saturated by induction on $\lbr{e\cap X}$. Let \mbox{$j\in[s(H)-1,r-1]$}, and assume that all edges $e$ such that $\lbr{e\cap X}\leq j$ are added to $F$, constituting the hypergraph $F_j\supseteq F$. Let us fix any $r$-edge $e$ such that $\lbr{e\cap X}=j+1$ and show that its addition to $F_j$ creates a copy of $H$. This will complete the proof.

By the definition of $s(H)$, there exists a set $S \subseteq V(H)$ of size $s(H)$ that is contained in exactly one edge $U \in E(H)$.

Construct an injection $f: V(H) \to V(F_j)$ such that $f(V(H)) \subseteq e \cup Z$, $f(U) = e$, and $f(S)\subseteq X$. The later is possible since $\lbr{e\cap X} \ge s(H)$.

Now, given that any edge $\wt{U} \in E(H)$ distinct from $U$ does not fully contain $S$ and
$X \cap f(V(H)) \subseteq  e = f(U)$, it follows that $\lbr{f(\smash{\wt{U}})\cap X} < \lbr{f(U)\cap X}=\lbr{e\cap X}=j+1$. Therefore, the only missing edge in $f(E(H))$ is $e$, and adding it creates a new copy of $H$, as required. \qedhere

\end{proof}

\begin{thebibliography}{99}
    \bibitem{Tuza91}  {\sc Z. Tuza},  Asymptotic growth of sparse saturated structures is locally determined. {\sl Discrete Math.} {\bf 108}:1-3 (1992) 397--402. Topological, algebraical and combinatorial structures. Frol\'{i}k's memorial volume.

    \bibitem{Neumann} {\sc J. von Neumann}, {\sl Theory of Self-Reproducing Automata}, Univ. Illinois Press, Urbana, 1966.

    \bibitem{Ulam} {\sc  S. Ulam}, Random processes and transformations, {\sl Proc. Internat. Congr. Math.} (1950) 264--275.

    \bibitem{Bollobas68} {\sc  B. Bollob\'{a}s},   Weakly  $k$-saturated graphs, Beitr\"{a}ge zur Graphentheorie   (Kolloquium, Manebach, 1967), Teubner, Leipzig, 1968, pp. 25--31.

    \bibitem{Adler} {\sc J. Adler}, {\sc U. Lev}, Bootstrap percolation: visualizations and applications, {\sl Braz. J. Phys.} {\bf 33} (2003) 641--644.

    \bibitem{Fontes} {\sc L.R. Fontes}, {\sc R.H. Schonmann}, {\sc V. Sidoravicius}, Stretched exponential fixation in stochastic Ising models at zero temperature, {\sl Comm. Math. Phys.}  {\bf 228} (2002) 495--518.

    \bibitem{Morris} {\sc R. Morris}, Zero-temperature Glauber dynamics on $\mathbb{Z}^d$, {\sl Probab. Theory Related Fields}  {\bf 149} (2011) 417--434.

    \bibitem{Frankl82} {\sc P. Frankl}, An extremal problem for two families of sets,  {\sl European J. Combin.}  {\bf 3} (1982) 125--127.

    \bibitem{Kalai84} {\sc  G. Kalai}, Weakly saturated graphs are rigid,  in: Convexity and graph theory (Jerusalem, 1981), North-Holland Math. Stud., 87, Ann. Discrete Math., 20, North-Holland, Amsterdam, 1984, pp. 189--190.

    \bibitem{Kalai85} {\sc  G. Kalai}, Hyperconnectivity of graphs,  {\sl Graphs Combin.}   {\bf 1} (1985) 65--79.

    \bibitem{KMM} {\sc G. Kronenberg}, {\sc T. Martins}, {\sc N. Morrison}, Weak saturation numbers of complete bipartite graphs in the clique, {\sl J. Combin. Theory Ser. A} {\bf 178} (2021) 105357.


    \bibitem{TerekhovZhukovskii} {\sc N. Terekhov}, {\sc M. Zhukovskii}, Weak saturation in graphs: A combinatorial approach, {\sl J. Combin. Theory Ser. B} {\bf 172} (2025) 146--167.

    \bibitem{AscoliHe} {\sc R. Ascoli}, {\sc X. He}, Rational values of the weak saturation limit, arXiv:2501.15686, 2025.

    \bibitem{Alon85} {\sc N. Alon}, An extremal problem for sets with applications to graph theory, {\sl J. Combin. Theory Ser. A}   {\bf 40} (1985), 82--89.

    \bibitem{Shapira23} {\sc A. Shapira}, {\sc M. Tyomkyn}, Weakly saturated hypergraphs and a conjecture of Tuza, {\sl Proc. Amer. Math. Soc.} {\bf 151} (2023), 2795--2805.

    \bibitem{Rodl85} {\sc V. R\"odl}, On a packing and covering problem, {\sl Eur. J. Combin.} {\bf 6} (1985), 69--78.
\end{thebibliography}
\end{document}